\documentclass[10pt,a4paper]{article}
\usepackage[utf8]{inputenc}
\usepackage[T1]{fontenc}
\usepackage{amsmath}
\usepackage{amsfonts}
\usepackage{amssymb}
\usepackage{amsthm}
\usepackage{graphicx}
\usepackage{authblk}

\newtheorem{definition}{Definition}
\newtheorem{proposition}{Proposition}
\newtheorem{corollary}{Corollary}
\newtheorem{theorem}{Theorem}
\newtheorem{lemma}{Lemma}

\author[1]{Matthieu Porte\thanks{matthieu.porte@yahoo.fr}\thanks{This work was conducted in 2017 as a MSc thesis at UPMC, under the direction of V. Baladi  (CNRS / IMJ-PRG), whom I thank for her guidance and her useful remarks at all stages of this work. I also thank P-A. Guihéneuf (UMPC / IMJ-PRG) for his helpful comments.}}
\date{}

\affil[1]{Département de Mathématiques et Applications, École Normale Supérieure (Paris)}

\begin{document}
	\title{Linear response for Dirac observables of Anosov diffeomorphisms}
	\maketitle
	
	\begin{abstract}
		We consider a $\mathcal{C}^3$ family $t\mapsto f_t$ of $\mathcal{C}^4$ Anosov diffeomorphisms on a compact Riemannian manifold $M$. Denoting by $\rho_t$ the SRB measure of $f_t$, we prove that the map $t\mapsto\int \theta  d\rho_t$ is differentiable if $\theta$ is of the form $\theta(x)=h(x)\delta(g(x)-a)$, with $\delta$ the Dirac distribution, $g:M\rightarrow \mathbb{R}$ a $\mathcal{C}^4$ function, $h:M\rightarrow\mathbb{R}$ a $\mathcal{C}^3$ function and $a$ a regular value  of $g$. We also require a transversality condition, namely that the intersection of the support of $h$ with the level set  $\{g(x)=a\} $ is foliated by 'admissible stable leaves'.
	\end{abstract}

	\section{Introduction}

		Consider a physical system described by a state $x$ on a compact Riemannian manifold $M$, whose evolution is given by a smooth discrete-time dynamical system $f$ on $M$. In order to study the asymptotic behavior of the system, one is often interested in the asymptotic mean value of an \textit{observable} of the system, that is of a function $\Phi: M \rightarrow\mathbb{R}$, and thus in studying the quantity \begin{equation}\label{birkhoff_sum}\dfrac{1}{n}\sum_{k=0}^{n-1} \Phi(f^k(x))\end{equation} and its limit as $n\rightarrow +\infty$.
		
		In the context of chaotic dynamics, on an actual physical system subject to uncertainties in the measurement of the state of the system, one is generally unable to compute explicitly the orbit of $x$. Yet, ergodic theory \cite{walters_book} allows one to further study (\ref{birkhoff_sum}). If $\rho$ is an ergodic $f$-invariant probability and if $\Phi\in L^1(d\rho)$, then, by Birkhoff's ergodic theorem, for $\rho$-almost all $x\in M$ :
		\begin{equation}
			\label{ergodic_theorem}
			\lim_{n\rightarrow\infty}\dfrac{1}{n}\sum_{k=0}^{n-1} \Phi(f^k(x)) = \int_M \Phi d\rho \ .
		\end{equation}
		
		Therefore, one may think of $\rho$ as the asymptotic state of the system starting from $\rho$-almost every point, subject to $f$. 
		Yet, a given dynamical system $f$ may have multiple ergodic invariant measures. 
		
		This leads to the following question: what are 'natural' invariant measures representing the state of our system ?
		
		Since $M$ is a Riemannian manifold, Lebesgue measure on $M$ is especially important: sets with Lebesgue positive measure are sets one may physically observe. 
	    The set of points such that (\ref{ergodic_theorem}) holds is called the \textit{basin of attraction} of $\rho$.
	    One way to answer our latter question is thus to require that the ergodic measure we investigate has a basin of attraction of full --- or at least positive --- Lebesgue measure. This line of thought led to the notion of \textit{SRB measure} --- short for Sinai-Ruelle-Bowen, who characterized this measure in the late 1960s and early 1970s. See for example \cite{ruelle_76} for an historical definition, or \cite{srb_young} for a contemporary review. SRB measures play a key role in non-equilibrium statistical mechanics where they represent non-equilibrium states of a system \cite{ruelle09}. We define formally SRB measures in Section \ref{transfer_operator_ss} in the context of Anosov diffeomorphisms.
	    
	    By the ergodic theorem, $\rho$ is SRB (in the sense given in Theorem \ref{srb_theorem} of section \ref{transfer_operator_ss}) if it is ergodic and absolutely continuous with respect to Lebesgue measure, yet there are other measures for which this definition holds. In the case of hyperbolic dynamics, this notion is especially useful, since one may not generally find an absolutely continuous ergodic measure.
	    
	    The framework of linear response aims to describe what happens to asymptotic states when the system is subject to a small perturbation. Consider that the system is subject to a durable perturbation. That is to say, that there is a small vector field $X$ on $M$ such that for each $n\geq 0,$ $x_n=f(x_{n-1}) + X(f(x_{n-1}))$, where $x_n$ is the perturbed orbit. What happens to the SRB measure of our system upon such a perturbation of the dynamics $f$?
	    
	    This question is of practical interest. For example, in the field of climate science, one is interested in evaluating the evolution of temperature across the globe under a durable perturbation of the concentration of carbonic dioxide in the atmosphere. See \cite{lucariniclimate17} for a contemporary example.
	    
	    Over the last decades, substantial progress has been made on a functional approach to linear response.
	    The functional approach takes roots in the following observation: physical invariant measures may be seen as the fixed points of an operator acting on functional spaces, the Ruelle-Perron-Frobenius transfer operator, formally defined by:
	    \begin{equation}
		    \mathcal{L}\phi(x)=\sum_{f(y)=x} \dfrac{\phi(y)}{|\mathrm{Det}D_yf|}.
	    \end{equation}
	    
	    Therefore, one is led to construct adapted Banach spaces for $\mathcal{L}$ to act on. An ideal space would yield a spectrum composed of $1$ as a simple maximal isolated eigenvalue, without any other peripheral eigenvalue, and with its spectrum outside a disc of radius $r<1$  constituted only of eigenvalues of finite multiplicity (the latter property being called a \textit{spectral gap} for $\mathcal{L}$). In \cite{baladiicm} is an illustration of how those spectral properties lead to linear response. 
	    
	    In the case of expanding dynamics, $\mathcal{L}$ has a smoothing effect and one may study $\mathcal{L}$ as an operator on spaces $\mathcal{C}^r(M)$ of smooth functions \cite{baladibook94, baladiicm}. In the case of hyperbolic dynamics, $\mathcal{L}$ no longer has a spectral gap on those spaces and one is led to consider spaces of anisotropic distributions for $\mathcal{L}$ to have a spectral gap. In \cite{gl06}, a geometric approach based on strictly invariant cones in the tangent space is developed, while in \cite{baladi_tsuji} a microlocal approach is developed, generalizing the usual Sobolev spaces. A review of those may be found in \cite{baladi17} or in the recent book \cite{baladibook16}. 
	    
	    If one manages to compute the expansion to first order of the SRB measure in a suitable Banach space $\mathcal{B}$ under a perturbation of $f$, then linear response holds for observables in the dual space $\mathcal{B}^*$.
	    
	    In \cite[§3]{bkl16}, the authors rely on the spaces $\mathcal{B}^{u,s}$ introduced by Gouëzel and Liverani in \cite{gl06} to prove linear response for a non-smooth observable. More precisely, differentiability of the map $t\mapsto\int h(x)\Theta(g(x)-a))d\rho_t$ is proven for a $\mathcal{C}^3$ family of $\mathcal{C}^4$ diffeomorphisms $f_t$ with transitive compact hyperbolic attractors, with $\rho_t$ the associated family of SRB measures,  and observables of the form $h(x)\Theta(g(x)-a)$ where $h:M\rightarrow\mathbb{R}$ is a $\mathcal{C}^3$ function, $\Theta$ is the Heaviside function and $g:M\rightarrow\mathbb{R}$ is a $\mathcal{C}^4$ function having $a\in\mathbb{R}$ as  a regular value. The assumption that  $\mathcal{W}_a = \{g(x)=a\}\cap\mathrm{supp}(h)$ is foliated by admissible stable leaves --- that is by submanifolds close to the stable manifold --- is also required. This latter assumption can be viewed as a transversality condition. The Heaviside function allows to study the response of regions of $M$ on level sets of $g$ above a given threshold.
	    
	    Under weaker assumptions --- that is, if $\beta\in(0,1)$, for a  $\mathcal{C}^{2+\beta}$ family of $\mathcal{C}^3$ diffeomorphisms with transitive compact hyperbolic attractors and if $g:M\rightarrow\mathbb{R}$ is $\mathcal{C}^2$ and $a\in\mathbb{R}$ is a regular value of $g$ --- the authors of \cite{bkl16} are still able to prove \textit{fractional response} for the observable $\theta(x) = \Theta(g(x) -a)$, that is to say that  $t\mapsto \int \theta d\rho_t$ is Hölder for some suitable Hölder exponent. Note that this result, weaker than linear response, does not require any transversality condition for the level set $\{g = a\}$. Thus it is typically valid for the level sets in the neighbourhood of a local maximum or minimum of $g$.
	    
	    The main result of this paper, Theorem \ref{linear_response}, extends the first of those results to the case of an observable of the form \begin{equation}\label{theta_form}h(x)\delta(g(x)-a) \ ,\end{equation} where $\delta$ is the Dirac distribution, without more regularity required. We also show linear response for observables based on derivatives of the Dirac distribution, up to increasing the regularity conditions. 
	    
	    To avoid some technical difficulties, we will limit ourselves to the case of transitive Anosov diffeomorphisms. The paper is organized as follows: in Section \ref{bus_spaces_s} we recall the definition of the spaces $\mathcal{B}^{u,s}$, the definition of the transfer operator, some of its properties on those spaces, and a result of differentiability of the SRB measure. In Sections \ref{dimension 2} and \ref{General case} we show that observables of the form (\ref{theta_form}) are in the dual of suitable $\mathcal{B}^{u,s}$ spaces and derive a corresponding linear response result. We first present the case where $M$ is of dimension 2 in Theorem \ref{linear_response_dim2}, which makes the argument simpler, before proceeding to the general case in Theorem \ref{linear_response}. Finally, in Section \ref{Applications} we illustrate our results by constructing examples of functions $g$ and $h$ for which linear response holds in the case of perturbations of a linear hyperbolic automorphism on the 2-torus --- the classical 'cat map'.

	\section{Transfer operators and the anisotropic spaces $\mathcal{B}^{u,s}$}
	\label{bus_spaces_s}
	
	Let $M$ be a compact Riemannian manifold of dimension $d\geq 2$. We first recall the definition of an Anosov diffeomorphism.
	
	\begin{definition}
		\label{Anosov_def}
		Let $f:M\rightarrow M$ be a $\mathcal{C}^1$ diffeomorphism of $M$. It is called an \textbf{Anosov diffeomorphism} if there is a splitting of the tangent space as a Whitney sum $TM = E^s \oplus E^u$, constants $0<\nu<1$ and $1<\lambda<\infty$ and a constant $c>0$ such that:
		\begin{itemize}
			\item $\forall x\in M, Df_xE^s(x)=E^s(f(x))$ and $Df_xE^u(x)=E^u(f(x))$
			\item $\forall w\in  E^s, \forall n > 0, \|Df^nw\|\leq c\nu^n\|w\|$
			\item $\forall v\in  E^u, \forall n > 0, \|Df^nv\|\geq c\lambda^n\|v\|$
		\end{itemize}
		$E^s$ and $E^u$ are respectively called the \textbf{stable} and \textbf{unstable bundles}, and their dimensions are respectively noted $d_s$ and $d_u$. If $x\in M$, $E^s(x)$ and $E^u(x)$ are called respectively the \textbf{stable} and \textbf{unstable directions} at $x$.
	\end{definition}
	
	Let $r\geq 3$ be a real number, $t\mapsto f_t$ a $\mathcal{C}^r$ family of $\mathcal{C}^{r+1}$ transitive Anosov diffeomorphisms, for $t\in (-\epsilon_0 ,\epsilon_0)=:I_0$. Let $X_t$ be the family of vector fields defined by $\partial_t f_t(x)= X_t \circ f_t(x)$, i.e. $X_t (x)= (\partial_t f_t) \circ f_t^{-1}(x)$ for all $x\in M$.
	Let $TM = E^s \oplus E^u$ be the splitting of the tangent bundle in stable and unstable directions for $f_0$, and $d_s, d_u$ respectively the stable and unstable bundles. Let $\lambda > 1$ be the weakest asymptotic expansion rate along the unstable directions and $\nu < 1$ the weakest asymptotic contraction rate along the stable directions for $f_0$.
		
	 Following the lines of \cite{bkl16}, our goal is to show linear response for Dirac observables and their derivatives.
	 
	\begin{definition}
	 		\label{dirac_observable_def}
	 		If $N\subset M$ is an embedded submanifold with boundary, the \textbf{Dirac distribution} $\delta_N$ on $N$ is the distribution acting on continuous functions by $$\langle \delta_N, \phi\rangle = \int_N \phi d\mu_N,$$ with $d\mu_N$ the Lebesgue measure on $N$.
	 		
	 		A \textbf{Dirac observable} is a distribution $h\delta_N$, with $h\in\mathcal{C}^r(M)$, acting on continuous functions by $$\langle h\delta_N, \phi\rangle = \int_N h\phi d\mu_N \ .$$	
	\end{definition}
	
	We will use the spaces $\mathcal{B}^{u,s}$, constructed in \cite{gl06},  adapted to the family $f_t$ (up to restricting our parameter to a smaller neighbourhood of $0$) . 
	
	\subsection{Admissible stable leaves}
		\label{admissible_stable_leaves_ss}
	In order to define the spaces $\mathcal{B}^{u,s}$, we recall from \cite{gl06} the definition of the set $\Sigma$ of admissible stable leaves. Those are small embedded submanifolds locally close to stable manifolds.
	
	Without loss of generality, the metric on $M$ is replaced by a metric \emph{à la Mather} \cite{mather68} and we hence assume that $Df_0$ has expansion rate stronger than $\lambda$, contraction rate lower than $\nu$ and that the angle between stable and unstable directions is everywhere arbitrarily close to $\pi / 2$.
	For small enough $\kappa >0$, we define the stable cone at $x\in M$ by
	\begin{equation}
		\mathcal{C}^s(x) = \{w+v\in T_xM ~ | ~ w\in E^s_x, v\perp E^s_x, \|v\|\leq \kappa\|w\|\}.	
	\end{equation}
	
	If $\kappa$ is small enough, $D_xf_0^{-1}(\mathcal{C}^s(x))\setminus\{0\}$ belongs to the interior of $\mathcal{C}^s(f_0^{-1}(x))$ and $D_xf_0^{-1}$ expands the vectors in $\mathcal{C}^s(x)$ by $\nu^{-1}$.
	
	There exist real numbers $\tau_i\in(0,1)$ and $\mathcal{C}^{r+1}$ coordinate charts $\psi_1,...,\psi_N$ with $\psi_i$ defined on $(-\tau_i,\tau_i)^d\subset \mathbb{R}^d$ (with its Euclidean norm) such that $M$ is covered by the open sets $\psi_i((-\tau_i/2,\tau_i/2)^d)_{i=1...N}$, and the following conditions hold :
	\begin{itemize}
		\item $D\psi_i(0)$ is an isometry;
		\item $D\psi_i(0).(\mathbb{R}^{d_s}\times\{0\}) = E^s(\psi_i(0)) \ $;
		\item The $\mathcal{C}^{r+1}$ norms of $\psi_i$ and its inverse are bounded by $1 + \kappa$.
	\end{itemize}

	The $\psi_i$ are called \emph{admissible charts}.
	
	We then may pick $c_i\in(\kappa , 2\kappa)$ such that the corresponding stable cone in charts
	
	\begin{equation*}
		\mathcal{C}_i^s = \left\{w+v\in\mathbb{R}^d ~| ~ w\in\mathbb{R}^{d_s}\times\{0\}, v\in\{0\}\times\mathbb{R}^{d_u}, \|v\|\leq c_i\|w\|\right\}
	\end{equation*}
	satisfies $D_x\psi_i(\mathcal{C}_i^s)\supset\mathcal{C}^s(\psi_i(x))$ and $D_{\psi_i(x)}f_t^{-1}(D\psi_i(x)\mathcal{C}_i^s)\subset\mathcal{C}^s(f_t^{-1}(\psi_i(x)))$ for any $x\in(-\tau_i,\tau_i)^d$ and all $t\in I_0$, up to restricting our interval $I_0$.
	
	Let $G_i(K)$ be the set of $\mathcal{C}^{r+1}$ maps $\chi : U_{\chi} \rightarrow (-\tau_i,\tau_i)^{d_u}$ defined on a subset $U_{\chi}$ of $(-\tau_i,\tau_i)^{d_s}$, with $|D\chi|<c_i$ and $|\chi|_{\mathcal{C}^{r+1}}\leq K$. In particular, the tangent space to the graph of $\chi$ belongs to the interior of the cone $\mathcal{C}_i^s$.
	
	As shown in \cite{gl06}, uniform hyperbolicity of $f_0$ implies that if $K$ is large enough, then there exists $K'<K$ such that for any $W\in G_i(K)$ and any $1\leq j\leq N$, we have $\psi_j^{-1}(f_t^{-1}(\psi_i(W)))\in G_i(K')$, for all $t\in I_0$, restricting $I_0$ again if necessary.
	
	Furthermore, we assume $\kappa$ is small enough so that $\nu(1+\kappa)^2\sqrt{1+4\kappa^2} < 1$. Let $K_1>1$ be such that $K_1 > 1+K_1\nu(1+\kappa)^2\sqrt{1+4\kappa^2}$
	
	\begin{definition}
	\label{admissible_stable_leaves}
	An \textbf{admissible graph} is a map $\chi$ defined on a ball 
	\begin{equation*}
		\overline{B}(w,K_1\delta)\subset (-2\tau_i/3,2\tau_i/3)^{d_s}
	\end{equation*}	
	for some $\delta >0$ such that $6K_1\delta < \min_i(\tau_i)$ and some $w\in (-2\tau_i/3,2\tau_i/3)^{d_s}$, with $\mathrm{range}(\mathrm{Id_{\mathbb{R}^{d_s}}},\chi)\in G_i(K)$ and taking its values in $(-2\tau_i/3,2\tau_i/3)^{d_u}$ .
	
	The set of \textbf{admissible stable leaves} is 
	\begin{align*}
	\Sigma := &\left\{\psi_i\circ(\mathrm{Id_{\mathbb{R}^{d_s}}},\chi)(\overline{B}(w,\delta)) ~| ~ \chi: \overline{B}(w,K_1\delta)\rightarrow \mathbb{R}^{d_u} \right.\\
	& \left. \text{ is an admissible graph on } (-2\tau_i/3,2\tau_i/3)^{d_s}\right\} \ .
	\end{align*}
	\end{definition}

	\subsection{Norm on $\mathcal{B}^{u,s}$}
	 Recall that $r\geq 3$.
	 If $W\in\Sigma$, let $\mathcal{V}^r(W)$ be the set of $\mathcal{C}^r$ vector fields on a neighbourhood of W. For $q\leq r$, let $\mathcal{C}_0^q(W)$ be the set of $\mathcal{C}^q$ functions on $W$ that vanish on a neighbourhood of the boundary of $W$.
	
	If $u\in\mathbb{N}$ and $s>0$ are such that $u+s<r$ and if $\phi\in\mathcal{C}^{u}$, we define 
	\begin{align*}
		\|\phi\|_{\mathcal{B}^{u,s}} &= \sup_{0\leq p\leq u} \|\phi\|_{p, s+p} \\
		&= \sup_{0\leq p\leq u} \sup_{W\in\Sigma}\sup_{\substack{v_1,...,v_p\in\mathcal{V}^r(W)\\ |v_i|_{\mathcal{C}^r\leq 1}}}\sup_{\substack{h\in\mathcal{C}_0^{p+s}(W) \\ |h|_{\mathcal{C}^{p+s}}\leq 1}} |\int_W (\partial_{v_1}...\partial_{v_p}\phi). h |.
	\end{align*}
	
	The space $\mathcal{B}^{u,s}$ is the completion of $\mathcal{C}^r(M)$ for the norm $\|.\|_{\mathcal{B}^{u,s}}$. It is straightforward from the definition of the norm on $\mathcal{B}^{u,s}$ that it is actually a space of distributions of order at most $s$. In Section 4 of \cite{gl06}, it is shown that this injection is continuous.
	
	\subsection{Transfer operator and spectral gap}
	\label{transfer_operator_ss}
	For $t\in I_0$, we define the transfer operator $\mathcal{L}_t$ on $\mathcal{C}^r(M)$ by the following formula, with $\phi \in \mathcal{C}^r(M)$ and $x\in M$
	
	\begin{equation*}
		\mathcal{L}_t\phi(x) = \frac{\phi\circ f_t^{-1}(x)}{|\mathrm{Det}Df_t(f_t^{-1}(x))|}.
	\end{equation*}
	where $\mathrm{Det}Df_t$ is the Jacobian determinant of $f_t$ with respect to the Lebesgue measure.
	
	By the change of variables formula, for $\phi \in\mathcal{C}^r(M)$ and $\psi \in \mathcal{C}^0(M)$ :
	\begin{align*}
	\langle  \psi, \mathcal{L}\phi\rangle &= \int_M \mathcal{L}\phi(x) \psi(x)dx= \int_M \phi(y) \psi\circ f(y) dy= \langle \psi\circ f, \phi\rangle.
	\end{align*}
	with $dx$ and $dy$ denoting the Lebesgue measure on $M$.	
	
	We recall the definition of the essential spectral radius of an operator.
	\begin{definition}
	Let $\mathcal{L}:\mathcal{B}\rightarrow\mathcal{B}$ be a bounded operator on a Banach space. The \textbf{essential spectral radius} of $\mathcal{L}$ (on $\mathcal{B}$), denoted $r_{ess}(\mathcal{L}|_{\mathcal{B}})$ (or $r_{ess}(\mathcal{L})$ if there is no ambiguity), is the smallest real number $\tau$ such that the spectrum of $\mathcal{L}$ outside of the disc of center $0$ and radius $\tau$ consists of isolated eigenvalues of finite multiplicity.
	\end{definition} 
	That is to say: up to the spectrum of modulus smaller than the essential spectral radius, the operator $\mathcal{L}$ behaves 'as' a compact operator.
	
	We state the definition of the SRB measure of a transitive Anosov diffeomorphism.
	\begin{theorem}\cite{srb_young}
		\label{srb_theorem}
		Let $M$ be a compact Riemannian manifold and let $f:M\rightarrow M$ be a $\mathcal{C}^2$ transitive Anosov diffeomorphism on $M$. Then there is a unique  $f$-invariant measure $\rho$ characterized by the following equivalent conditions:
		\begin{itemize}
			\item $\rho$ has absolutely continuous (w.r.t Lebesgue measure) conditional measures on unstable manifolds;
			\item There is a set $V\subset M$ of full Lebesgue measure such that for every $x\in V$, in the weak-$\ast$ topology: 
			\begin{equation}
			\lim_{n\rightarrow\infty} \dfrac{1}{n}\sum_{k=0}^{n-1}\delta_{f^k(x)} = \rho \ .
			\end{equation}
		\end{itemize}
		This measure is called the \textbf{SRB measure} of $f$.
	\end{theorem}

	We recall several facts about the spaces $\mathcal{B}^{u,s}$ we will use in the next proposition. All of them are proved in \cite{gl06}, using that $f_t$ is an Anosov transitive diffeomorphism and thus mixing. Thus, we will not prove this proposition, but we provide and admit intermediate results needed for the reader to prove it.
	\begin{proposition}
		\label{bus_prop}\cite{gl06}
		Let $r\geq 3$ be a real number.
		Let $0<u+s<r$, with $s>0$ and $u\in\mathbb{N}$.		
		The transfer operator $\mathcal{L}_t$ admits a unique bounded extension to $\mathcal{B}^{u,s}$. It has a spectral radius of $1$. 
		Assume further that $u\geq 1$.	Then $\mathcal{L}_t$ has essential spectral radius strictly smaller than $1$.
		The only eigenvalue of $\mathcal{L}_t$ of modulus $1$ is $1$ ; it is a simple eigenvalue. Its eigenvector $\rho_t \in \mathcal{B}^{u,s}$ is actually a nonnegative measure, which is the SRB measure of $f_t$.
	\end{proposition}

	The following result is based on an application of the Nussbaum formula \cite{nussbaum69} for the essential spectral radius
	\begin{lemma}(Hennion's theorem) \cite{hennion93}
		\label{nussbaum}
		Let	$(\mathcal{B},\|.\|)$ and $(\mathcal{B}',\|.\|')$ be two Banach spaces such that $\mathcal{B}\hookrightarrow\mathcal{B}'$ compactly. Let $\mathcal{L}:\mathcal{B}\rightarrow \mathcal{B}$ be a bounded operator (for the norm $\|.\|$).
		Assume that there exist two sequences of real numbers $\beta_n$ and $B_n$ such that, for any $n\geq 1$ and any $\phi\in\mathcal{B}$
		\begin{equation}
			\|\mathcal{L}^n\phi\|\leq \beta_n\|\phi\| + B_n\|\phi\|' .
		\end{equation}
		Then the essential spectral radius of $\mathcal{L}$ on $\mathcal{B}$ is not larger than $\underset{n\rightarrow\infty}{\mathrm{liminf}}(\beta_n)^{1/n}$.
	\end{lemma}
	
	\begin{lemma} (Compact injection) \cite{gl06}
		\label{compact_bus}
		Let $0<u+s<r$ with $u\in\mathbb{N}$, $u\geq 1$, $s>0$. The canonical injection $\mathcal{B}^{u,s}\hookrightarrow\mathcal{B}^{u-1,s+1}$ is compact.
	\end{lemma}
	
	\begin{lemma}(Lasota-Yorke inequality) \cite{gl06}
		\label{lasotayorke}
		For each $u\in\mathbb{N}$ and $s>0$ with $0<u+s<r$, there exist $A_{u,s}>0$ such that for each $n\in\mathbb{N}$ and $\phi\in\mathcal{B}^{0,s}$:
		\begin{equation}
			\|\mathcal{L}^n\phi\|_{\mathcal{B}^{0,s}} \leq A_{0,s}\|\phi\|_{\mathcal{B}^{0,s}} ;
		\end{equation}
	and, if $u\geq 1$, there exists a sequence $C_{u,s}(n)>0$ such that for all $\phi \in \mathcal{B}^{u,s}$:
		\begin{equation}
			\|\mathcal{L}^n\phi\|_{\mathcal{B}^{u,s}}\leq A_{u,s} \max(\lambda^{-u}, \nu^{s})^n\|\phi\|_{\mathcal{B}^{u,s}} + C_{u,s}(n)\|\phi\|_{\mathcal{B}^{u-1,s+1}} .
		\end{equation}
	\end{lemma}
	
	\subsection{Linear response}

	The following theorem, stated in \cite[Theorem 2.7]{gl06} is the basis for all our further linear response results.
	\begin{theorem}
		\label{linear_response_measure}
		Let $r\geq 3$.
		Let $u\in\mathbb{N}$, $s>0$ with $u\geq 2$ and $0<u+s<r$.
		Then the map $t\mapsto \rho_t\in\mathcal{B}^{u,s}$ is differentiable in $\mathcal{B}^{u-2,s+2}$, and for $t\in I_0$ we have :
		\begin{align}
			\label{derivative_rho}
			\partial_t\rho_t &= -(1-\mathcal{L}_t)^{-1}\mathrm{div}(\rho_tX_t)\\
			\label{sum_derivative_rho}
			&= -\sum_{k=0}^\infty \mathcal{L}_t^k\mathrm{div}(\rho_tX_t) \ .
		\end{align}
		In particular, we claim that $(1-\mathcal{L}_t)$ is an invertible automorphism of the space  $\{\phi\in\mathcal{B}^{u-2,s+2},  \int\phi dx=0\}$ which is the kernel of the spectral projector associated to the eigenvalue $1$.
		
	\end{theorem}
		The reader should note that we only assumed $r\geq 3$, thus the previous theorem applies to $\mathcal{C}^4$ families of $\mathcal{C}^3$ diffeomorphisms.
	
	\begin{proof}
	We will only prove this result under the stronger assumptions that $r\geq 4$ and $u\geq 3$. This allows us to give a simpler proof, using that $\mathcal{L}_0$ has a spectral gap both on $\mathcal{B}^{u,s}$ and $\mathcal{B}^{u-2, s+2}$. We refer to \cite[Theorem 2.7]{gl06} for a full proof of the general case. 
	
	We first state and prove the following lemma, which is used to control the norm of derivatives of distributions in $\mathcal{B}^{u,s}$.
		\begin{lemma}
			\label{derivative_control}
			Let $u\in\mathbb{N}$, $s>0$ with $u\geq 1$ and $0<u+s<r$. Let $Y$ be a $\mathcal{C}^{r}$ vector field over $M$. Then
			$\partial_Y: \left\{
			\begin{array}{rcl}
			\mathcal{B}^{u,s} &\rightarrow& \mathcal{B}^{u-1,s+1}\\
			\phi &\mapsto& \partial_Y\phi
			\end{array}\right.$
		is a bounded operator.			
		\end{lemma}
		\begin{proof}
			Let $u,s, Y$ be as stated. Let $\phi\in\mathcal{C}^r(M)$
			By definition of the norms $\|.\|_{p,s+p}$ we have, for $0\leq p \leq u-1$
			\begin{align}
				\|\partial_Y \phi\|_{p,s+1+p} &\leq |Y|_{\mathcal{C}^r}\|\phi\|_{p+1,s+1+p} \\
				&\leq |Y|_{\mathcal{C}^r}\|\phi\|_{\mathcal{B}^{u,s}} ~.
			\end{align}
			Thus
			\begin{equation}
				\|\partial_Y\phi\|_{\mathcal{B}^{u-1,s+1}} \leq |Y|_{\mathcal{C}^r}\|\phi\|_{\mathcal{B}^{u,s}}  ~.
			\end{equation}
			By density, $\partial_Y$ extends to a bounded operator $\partial_Y : \mathcal{B}^{u,s}\rightarrow\mathcal{B}^{u-1,s+1}$.
			
		\end{proof}
		
		The following lemma allows us to control the norm of a product of an anisotropic distribution with a smooth function.
		
		\begin{lemma}
			\label{product_control}
			Let $k\in\mathbb{N}$ with $0<k\leq r$ and $h\in\mathcal{C}^k(M)$. Let $u\in\mathbb{N}$, $s>0$ with $u\geq 1$ and $0<u+s<k$. Then $m_h: \left\{\begin{array}{rcl}
			\mathcal{B}^{u,s} &\rightarrow& \mathcal{B}^{u,s}\\
			\phi &\mapsto& h\phi\end{array}\right.$ is a bounded operator.
		\end{lemma}
		
		\begin{proof}
			For $\phi\in\mathcal{C}^k$, up to considering an equivalent norm on $\mathcal{C}^k(M)$, there exists a constant $C>0$ such that:
			$$|h\phi|_{\mathcal{C}^k}\leq C|h|_{\mathcal{C}^k}|\phi|_{\mathcal{C}^k}$$
			
		The result follows from the inclusion $\mathcal{C}^k(M)\hookrightarrow \mathcal{B}^{u,s}$ and from the fact that $\mathcal{B}^{u,s}$ may be obtained by completion of $\mathcal{C}^{k}(M)$.	
		\end{proof}
		
		We now prove differentiability of the transfer operator
		\begin{lemma}
			\label{diff_transfer}
			Let $u,s$ be as stated in Theorem \ref{linear_response_measure}. Then $t\mapsto\mathcal{L}_t$ is differentiable as a family of bounded operators from $\mathcal{B}^{u,s}$ to $\mathcal{B}^{u-2,s+2}$  --- that is, after composition with the canonical inclusion. Moreover, for $\phi\in\mathcal{B}^{u,s}$,
			\begin{equation}
				\mathcal{M}_t \phi :=\partial_t\mathcal{L}_t\phi
				=-\mathrm{div}((\mathcal{L}_t\phi) X_t)
			\end{equation}
			and $\mathcal{M}_t : \mathcal{B}^{u,s}\rightarrow\mathcal{B}^{u-2,s+2}$ is a bounded operator.
		\end{lemma}
		\begin{proof}
			Let $\phi\in\mathcal{C}^r(M)$. Without loss of generality, we prove the result for $t=0$.
			By Lemmas \ref{derivative_control} and \ref{product_control}, since $X_0$ is $\mathcal{C}^{r-1}$, we have $\mathrm{div}((\mathcal{L}_0\phi)X_0)\in\mathcal{B}^{u-2,s+2}$.
			By a straightforward computation, for $t\in I_0$ and $x\in M$ :
			\begin{align*}
				\mathcal{L}_t\phi(x) - \mathcal{L}_0\phi(x) &= t\left(-\langle X(x), \nabla(\mathcal{L}_0\phi)(x)\rangle - \mathrm{div}(X)(x)\mathcal{L}_t\phi\right)\\
				&  + O\left(t^2 (|\phi(x)| + |\nabla\phi(x)| + |\nabla^2\phi(x)|)\right)
			\end{align*}
			Thus, by Lemma \ref{derivative_control} applied twice to control $|\nabla\phi|$ and $|\nabla^2\phi|$:
			\begin{equation}
				\|\mathcal{L}_t\phi - \mathcal{L}_0\phi\  - t(-\mathrm{div}((\mathcal{L}_0\phi) X_0))\|_{\mathcal{B}^{u-2,s+2}} = O\left(t^2 \|\phi\|_{\mathcal{B}^{u-2,s+2}}\right)  ~.
			\end{equation}
			Thus $\mathcal{L}_t\phi$ is differentiable in $\mathcal{B}^{u-2,s+2}$ and $\mathcal{M}_t$ extends by density to an operator on $\mathcal{B}^{u,s}$ that satisfies $\mathcal{M}_t = \partial_t\mathcal{L}_t$.
		\end{proof}
		
		From now on, the proof is standard and we follow the lines of \cite{baladiicm} to show (\ref{derivative_rho}) and (\ref{sum_derivative_rho}).
		Without loss of generality, we consider only differentiability of $\rho_t$ at $t=0$.
		Let $u,s$ be as stated in Theorem \ref{linear_response_measure}. Assume that $u\geq 3$.
		Since $\mathcal{L}_0$ has essential spectral radius strictly lower than $1$, we can find a simple closed curve in the complex plane $\gamma$ isolating $1$ from the rest of the spectrum in both $\mathcal{B}^{u,s}$ and $\mathcal{B}^{u-2,s+2}$.
		
		By semi-continuity of separated parts of the spectrum \cite[section IV.4]{kato} and continuity of the family $\mathcal{L}_t$ in both $\mathcal{B}^{u,s}$ and $\mathcal{B}^{u-2,s+2}$, up to restricting the interval for parameter $t$, the curve $\gamma$ separates the spectrum into two regions for every $\mathcal{L}_t$, and the spectrum inside $\gamma$ corresponds to an eigenspace of dimension exactly $1$. Thus $1$ is the only eigenvalue inside $\gamma$.
		
		For any $t\in I_0$: 
		\begin{equation}
			\rho_t = \dfrac{1}{2i\pi}\int_\gamma (z-\mathcal{L}_t)^{-1}\rho_0 dz \ .
		\end{equation}
		Since, by Lemma \ref{diff_transfer}, the map $t\mapsto\mathcal{L}_t$ is differentiable, we have that $\rho_t$ is differentiable in $\mathcal{B}^{u-2,s+2}$ and that
		\begin{align*}
		\partial_t\rho_t &= \dfrac{1}{2i\pi}\int_\gamma (z-\mathcal{L}_0)^{-1}\mathcal{M}_0(z-\mathcal{L}_0)^{-1}\rho_0 dz \\
		&= \dfrac{1}{2i\pi}\int_\gamma (z-\mathcal{L}_0)^{-1}\mathcal{M}_0\dfrac{\rho_0}{z-1}dz\\
		&= \dfrac{1}{2i\pi}(1-\mathcal{L}_0)^{-1}(\int_\gamma\dfrac{\mathcal{M}_0\rho_0}{z-1}dz - \int_\gamma (z-\mathcal{L}_0)^{-1}\mathcal{M}_0\rho_0 dz)\\
		&= (1-\mathcal{L}_0)^{-1}(1-\Pi_0)\mathcal{M}_0\rho_0\\
		&= -(1-\mathcal{L}_0)^{-1}(1-\Pi_0)\mathrm{div}(\rho_0X_0)
		\end{align*}
		again by Lemma \ref{diff_transfer} and with $\Pi_0$ the spectral projector of $\mathcal{L}_0$ on the eigenvalue $1$, which actually is integration with respect to Lebesgue measure --- i.e. $\Pi_0\phi=\int\phi dx$ for $\phi\in\mathcal{B}^{u-2,s+2}$.
		
		Since $M$ is boundaryless, integration by parts yields :
		\begin{equation}
		\partial_t\rho_t = -(1-\mathcal{L}_0)^{-1}(\mathrm{div}(\rho_0X_0)) \ .
		\end{equation}
		Since $\Pi_0(\mathrm{div}(\rho_0X_0)) =0$ and, by Proposition \ref{bus_prop}, the operator $\mathcal{L}_0$ has no other peripheral eigenvalue in $\mathcal{B}^{u-2,s+2}$, this derivative is the exponentially convergent sum
		\begin{equation}
		\partial_t\rho_t = -\sum_{k=0}^\infty \mathcal{L}_0^k(\mathrm{div}(\rho_0X_0)) \ .
		\end{equation}
	\end{proof}
	
	We now state a corollary, which is used to prove linear response. We first define the pullback of a distribution by $f_t$.
	\begin{definition}
		Let $u\in\mathbb{N}, s>0$ be such that $0<u+s<r$. Let $\theta\in(\mathcal{B}^{u,s})^*$. We define by duality the \textbf{pullback} of $\theta$ by $f_0$, written $f_t^*\theta$ :
		\begin{equation}
			\langle f_t^*\theta,\phi\rangle = \langle \theta, \mathcal{L}_t\phi\rangle, \forall \phi\in\mathcal{B}^{u,s} .
		\end{equation}
	\end{definition}
	By the change of variables formula, this definition corresponds to the pullback of an absolutely continuous measure with respect to the Lebesgue measure on $M$.
	
	\begin{corollary}
	\label{dual_lr}
	Let $u\in\mathbb{N}$, $s>0$ with $u\geq 2$ and $0<u+s<r$. Let $\theta\in(\mathcal{B}^{u-2,s+2})^*$.
	Then the map $t\mapsto\langle\theta, \rho_t\rangle$ is differentiable and
	\begin{equation}
		\partial_t\langle\theta, \rho_t\rangle = \sum_{k=0}^\infty \langle \nabla((f_t^k)^*\theta), \rho_t X_t\rangle \ .
	\end{equation}
	\end{corollary}
	
	\begin{proof}
		This is a direct application of Theorem \ref{linear_response}, Definition \ref{dirac_observable_def} and integration by parts.
		\begin{align*}
	\partial_t(\langle \theta, \rho_t\rangle) &= - \sum_{k=0}^\infty \langle \theta, \mathcal{L}_t^k\mathrm{div}(X_t\rho_t)\rangle= -\sum_{k=0}^\infty \langle (f_t^{*})^k\theta, \mathrm{div}(X_t\rho_t)\rangle \\
			&= \sum_{k=0}^\infty \langle \nabla ((f_t^{*})^k\theta), X_t\rho_t\rangle .
		\end{align*}
	\end{proof}

	\section{Linear response for Dirac observables in dimension 2}
	\label{dimension 2}
	Our aim is to show a linear response result for observables localized on the level sets of a function $g:M\rightarrow \mathbb{R}$ --- that is, of the form $\theta = h\delta_{\mathcal{W}_a}$ where $h:M\rightarrow\mathbb{R}$ is a smooth function,  $a\in\mathbb{R}$ is a regular value of $g$, also where $\mathcal{W}_a=\{x\in M, g(x)=a\}\cap\mathrm{supp}(h)$ and $\delta_{\mathcal{W}_a}$ is as defined in Definition \ref{dirac_observable_def} (since $a$ is a regular value of $g$ and if $h$ is smooth enough, $\mathcal{W}_a$ is indeed a submanifold with boundary of $M$).
	
	In Section \ref{General case}, we will show a more general result, but we first present the --- simpler --- argument for $M$ of dimension $2$, in which case both stable and unstable directions are of dimension $1$. While in the general case we will require $\mathcal{W}_a$ to be \textit{foliated} by admissible stable leaves, in this simpler setting we  require  that $\mathcal{W}_a$ is an admissible stable leaf.
	
	Assume that $\mathrm{dim}(M)=2$.
	
	\begin{proposition}
		\label{dim2_prop}
		Let $W\in\Sigma$, and $Y_1,..., Y_p\in\mathcal{V}^r(W)$ with $p\in\mathbb{N}$. Let $k>p$ and $h\in\mathcal{C}_0^k(W)$. Let also $0 < s < r-p$. 
		
		If $s+p\leq k$, then $\theta := h(\partial_{Y_1}...\partial_{Y_p}\delta_W) \in (\mathcal{B}^{p,s})^*$.
	\end{proposition}
	
	\begin{proof}
		Let $\phi\in\mathcal{C}^r(M)$.
		\begin{align*}
			|\langle \theta, \phi\rangle| &= |\int_W \partial_{Y_1}...\partial_{Y_p} (h \phi) d\mu_W| \\
			&\leq \sum_{\substack{I\subset \{1,...,p\}\\
				J = \{1,...,p\}\setminus I\\				
				i_1<...<i_m\in I\\
				j_1<...<j_{p-m}\in J}
				} | \int_W (\partial_{Y_{i_m}}...\partial_{Y_{i_1}}\phi) (\partial_{Y_{j_{p-m}}}...\partial_{Y_{j_1}} h)d\mu_W |\\
			&\leq \sum_{\substack{I\subset \{1,...,p\}\\
				J = \{1,...,p\}\setminus I\\				
				i_1<...<i_m\in I\\
				j_1<...<j_{p-m}\in J}} \mathrm{Vol}(W)(\sup_{1\leq i\leq p}|Y_i|_{\mathcal{C}^r})^m |\partial_{Y_{j_{p-m}}}...\partial_{Y_{j_1}} h|_{\mathcal{C}_0^{s+m}} \|\phi\|_{m,s+m}\\
			&\leq 2^p (\sup_{1\leq i\leq p}|Y_i|_{\mathcal{C}^r})^p \mathrm{Vol}(W) |h|_{\mathcal{C}_0^{k}} \|\phi\|_{\mathcal{B}^{p,s}} .
		\end{align*}
		Hence, by density, $\theta$ extends to a continuous linear form on $\mathcal{B}^{p,s}$.
	\end{proof}
	
	\begin{corollary}
		\label{level_set}
		Let $p\in\mathbb{N}$, $k>p$ and $h\in\mathcal{C}^k(M)$. Let $Y_1,...,Y_p$ be $\mathcal{C}^r$ vector fields over $M$. Let $g\in\mathcal{C}^{r+1}(M)$ and $a\in\mathbb{R}$ be such that $\mathcal{W}_a\in\Sigma$, with 
		\begin{equation}
		\mathcal{W}_a = \{x\in M | g(x)=a\} \cap V_h
		\end{equation}
		 and $V_h$ a neighbourhood of $\mathrm{supp}(h)$.
		 
		Let $0<s<r-p$. If $s+p\leq k$, then 
		\begin{equation}
		h(x)(\partial_{Y_1}...\partial_{Y_p}\delta_{\mathcal{W}_a})(x)=h(x)(\partial_{Y_1}...\partial_{Y_p}\delta (g(x)-a))\in(\mathcal{B}^{p,s})^*.
		\end{equation}
	\end{corollary}
	
	\begin{proof}
		This Corollary is a direct application of Proposition \ref{dim2_prop} with $W = \mathcal{W}_a$.
	\end{proof}
	
	We thus deduce a linear response result in dimension 2.

	\begin{theorem}
	\label{linear_response_dim2}
		Let $r\geq 3$, $\epsilon_0>0$, $t\mapsto f_t$ be, for $t\in(-\epsilon_0,\epsilon_0)$, a $\mathcal{C}^r$ family of $\mathcal{C}^{r+1}$ transitive Anosov diffeomorphisms  on a compact Riemannian manifold $M$ with $\mathrm{dim}(M)=2$. Let $X_t = (\partial_t f_t) \circ f_t^{-1}$.
		
		Let $h\in\mathcal{C}^r(M)$, $p\in\mathbb{N}$ be such that $p\leq r-3$ and $Y_1,...,Y_p$ be $\mathcal{C}^r$ vector fields on $M$. Let $g\in\mathcal{C}^{r+1}(M)$ and $a\in\mathbb{R}$ be such that $\mathcal{W}_a\in\Sigma$, with 
		\begin{equation}
		\mathcal{W}_a = \{x\in M | g(x)=a\} \cap V_h
		\end{equation}
		and $V_h$ a neighbourhood of $\mathrm{supp}(h)$.
		
		Let $\theta = h(\partial_{Y_1}...\partial_{Y_p}\delta_{\mathcal{W}_a}).$
		
		Then the map $t\mapsto\langle \theta, \rho_t\rangle$ is differentiable at $t=0$, and
		\begin{equation}
			\label{unmixing_lr}
			\partial_t(\langle \theta, \rho_t\rangle)|_{t=0} = - \langle \theta, (1-\mathcal{L}_0)^{-1}(\mathrm{div}(X_0\rho_0)\rangle \ .
		\end{equation}
		Furthermore, the derivative is the exponentially convergent sum :
		\begin{equation}
			\label{mixing_lr}
			\partial_t(\langle \theta, \rho_t\rangle)|_{t=0} = \sum_{k=0}^\infty \langle \nabla((f_0^{*})^k\theta), X_0\rho_0\rangle \ .
		\end{equation}
	\end{theorem}
	Thus, we require that $r\geq 3+p$ to obtain linear response for observables constructed with derivatives of the Dirac up to order $p$.
	\begin{proof}
		Let $u\geq 2+p$, $0<s<\mathrm{min}(1,r-u)$.
		
		By Theorem \ref{linear_response_measure}, the map $t\mapsto\rho_t$ is differentiable in $\mathcal{B}^{u-2,s+2}$ and
		\begin{equation}
		\partial_t \rho_t|_{t=0} = -(1-\mathcal{L}_0)^{-1}(\mathrm{div} (X_0\rho_0)) .
		\end{equation}
		By Corollary \ref{level_set}, since $r\geq p+3 > s+2+p$, we have $\theta\in(\mathcal{B}^{u-2, s+2})^*$, showing equality (\ref{unmixing_lr}).
		
		Finally, (\ref{mixing_lr}) follows from Corollary \ref{dual_lr}.
	\end{proof}
	
	The reader should note that, since $\mathcal{W}_a$ is a submanifold of codimension $1$, it is not true that $\langle\theta,\mathcal{L} \phi\rangle = \langle\theta\circ f, \phi\rangle$ where $\phi\in\mathcal{C}^0(M)$ : Lebesgue measure on $f^{-1}(\mathcal{W}_a)$ is not the image of Lebesgue measure on $\mathcal{W}_a$ by $f^{-1}$. That is: $(f^*\theta)(y) \neq h(f(y))\delta(g(f(y))-a)$. 
	
	\section{Linear response for Dirac observables in higher dimensions}
	\label{General case}
	Assume now $d\geq 2$ is general.	
	
	It should be noted that Theorem \ref{linear_response_dim2} actually applies whenever the unstable direction is of dimension $1$.
	
	Generally, the submanifold $\mathcal{W}_a$ has codimension $1$, and hence cannot be an admissible stable leaf. Yet, we can show a result similar to Proposition \ref{dim2_prop} for embedded submanifolds foliated by admissible stable leaves, which leads to an analogue of Theorem \ref{linear_response_dim2}.
	
	The current section has mostly the same structure as Section \ref{dimension 2}, but we adapt our arguments to manage the fact that $\mathcal{W}_a$ may only be foliated by admissible stable leaves.
	
	\begin{definition}
		Let $N\subset M$ be an embedded submanifold of dimension $d'\leq d$. We say that $N$ is \textbf{foliated by admissible stable leaves} if there is a $\mathcal{C}^{r+1}$ atlas $(U_i,\psi_i)_{i\in I}$ of a neighbourhood of $N$ such that, if $i\in I$, $\psi_i^{-1}(\mathbb{R}^{d'}\times \{0\})=N\cap U_i$ and $x_u\in\mathbb{R}^{d'-d_s}$, then $\psi_i^{-1}(\mathbb{R}^{d_s}\times\{x_u\}\times\{0\})$ is an admissible stable leaf.
	\end{definition}
	
	We generalize Proposition \ref{dim2_prop} :
	\begin{proposition}
		\label{dual_prop_foliated}
		Let $N\subset M$ be an embedded submanifold of dimension $d'$ and assume it is foliated by admissible stable leaves. Let $Y_1,...,Y_p$ be $\mathcal{C}^r$ vector fields defined on a neighbourhood of $N$. Let $k\in\mathbb{N}$ be such that $k > p$, and let $h\in\mathcal{C}_0^k(N)$. Let $s>0$ be such that $0<p+s<r$.
		
		If $s+p\leq k$, then $\theta = h(\partial_{Y_1}...\partial_{Y_p}\delta_N)\in(\mathcal{B}^{p,s})^*$.
	\end{proposition}
	
	\begin{proof}
		Let $(U_i,\psi_i)$ be a $\mathcal{C}^{r+1}$ atlas adapted to the foliation of $N$ and $\alpha_i$ an adapted partition of unity. Let $\phi\in\mathcal{C}^r(M)$, and $V_i = \psi_i(U_i\cap N)$. Then: 
		
		\begin{align*}
			|&\int_N \partial_{Y_1}...\partial_{Y_p}(\phi h)| = |\int_N \sum_{i\in I} \alpha_i \partial_{Y_1}...\partial_{Y_p}(\phi h) | \\
			&\leq \sum_{i\in I} |\int_{V_i} \frac{\alpha_i\partial_{Y_1}...\partial_{Y_p}(\phi h)}{|\mathrm{Det}D\psi_i|}\circ\psi_i^{-1}dx_1...dx_{d'}|\\
			&\leq \sum_{i\in I} \int_{(\mathbb{R}^{d'-d_s}\times\{0\})\cap V_i} \left[ \int_{(\mathbb{R}^{d_s}\times\{(x_u,0)\})\cap V_i}\frac{|\alpha_i\partial_{Y_1}...\partial_{Y_p}(\phi h)|}{|\mathrm{Det}D\psi_i|}\circ\psi_i^{-1}(x)dx_s \right] dx_u \\
			&\leq C \sum_{i\in I}  \mathrm{Vol}(U_i\cap N) \biggl|\frac{\alpha_i h}{|\mathrm{Det}D\psi_i|}\biggr|_{\mathcal{C}_0^k}\|\phi\|_{\mathcal{B}^{p,s}} \\
			&\leq C'\|\phi\|_{\mathcal{B}^{p,s}} \  .
		\end{align*}
		The penultimate inequality comes from Proposition \ref{dim2_prop} applied to each admissible stable leaf in the foliation.
		
		Hence $\theta$ extends by density to a continuous linear form on $\mathcal{B}^{p,s}$, and thus on $\mathcal{B}^{u,s}$.
	\end{proof}
	
	Note that, since $M$ is trivially foliated by admissible stable leaves, Proposition \ref{dual_prop_foliated} shows that the Lebesgue measure is in the dual of spaces $\mathcal{B}^{u,s}$, a fact used earlier to state that it was the fixed point of the dual operator $\mathcal{L}^*$.
	
	We can now state our main result:
	\begin{theorem}
		\label{linear_response}
		Let $r\geq 3$, $\epsilon_0>0$, $t\mapsto f_t$ be, for $t\in(-\epsilon_0,\epsilon_0)$, a $\mathcal{C}^r$ family of $\mathcal{C}^{r+1}$ transitive Anosov diffeomorphisms  on a compact Riemannian manifold $M$. Let $X_t = (\partial_t f_t) \circ f_t^{-1}$.
		
		Let $h\in\mathcal{C}^r(M)$. Let $p\in\mathbb{N}$ be such that $p\leq r-3$ and $Y_1,... Y_p$ be $\mathcal{C}^r$ vector fields on $M$.
		Let $g\in\mathcal{C}^{r+1}(M)$, $V_h$ a neighbourhood of $\mathrm{supp}(h)$, $a\in\mathbb{R}$ such that $\mathcal{W}_a = \{g(x)=a\}\cap V_h$ admits a $\mathcal{C}^{r+1}$ foliation by admissible stable leaves.
		
		Then the map $t\mapsto\langle h(\partial_{Y_1}...\partial_{Y_p}\delta)(g-a), \rho_t\rangle =: \langle \theta, \rho_t\rangle$ is differentiable at $t=0$ and
		\begin{equation}
			\partial_t(\langle \theta, \rho_t\rangle)|_{t=0} = - \langle \theta, (1-\mathcal{L}_0)^{-1}(\mathrm{div}(X_0\rho_0)\rangle \  .
		\end{equation}
		Furthermore, the derivative is the exponentially convergent sum
		\begin{equation}
			\partial_t(\langle \theta, \rho_t\rangle)|_{t=0} = \sum_{k=0}^\infty \langle \nabla((f_0^{*})^k\theta), X_0\rho_0\rangle \  .
		\end{equation}
	\end{theorem}
	\begin{proof}
		Follow the proof of \ref{linear_response_dim2}, replacing Proposition \ref{dim2_prop} by Proposition \ref{dual_prop_foliated}.
	\end{proof}

	\section{Applications}
	\label{Applications}

	In this section, we provide applications of Theorems \ref{linear_response_dim2} and \ref{linear_response} to a specific linear hyperbolic automorphism of the 2-torus $\mathbb{T}^2=\mathbb{R}^2/\mathbb{Z}^2$, the so-called 'cat map'.

	Let \begin{equation}
	F:\left\{\begin{array}{ll}
	\mathbb{R}^2\rightarrow\mathbb{R}^2\\
	(x,y)\mapsto(2x+y, x+y)\end{array}\right.
	\end{equation}
	
	and $f:\mathbb{T}^2\rightarrow\mathbb{T}^2$ be the induced quotient map on the torus.
	
	Let $\pi : \mathbb{R}^2\rightarrow\mathbb{T}^2$ be the projection on the torus.
	
	It is well-known \cite[§1.8, 6.4]{kh_book} that $f$ is a smooth transitive Anosov diffeomorphism, with unstable and stable manifolds all obtained by translation of those of the point $(0,0)$, an expansion rate of $\lambda = \dfrac{3+\sqrt{5}}{2}$ in the direction given by $\left(\begin{array}{cc}1\\\dfrac{\sqrt{5}-1}{2}\end{array}\right)$ and a contraction rate of $\nu=\lambda^{-1}=\dfrac{3-\sqrt{5}}{2}$ in the direction given by $\left(\begin{array}{cc}1\\\dfrac{-\sqrt{5}-1}{2}\end{array}\right)$. Throughout this section, we will note $(s,u)$ the coordinates in $\mathbb{R}^2$ in the direct orthonormal frame given by the stable and unstable directions.
	
	 Furthermore, since $\mathrm{Det}F=1$ everywhere, $f$ is volume-preserving. Hence its SRB measure is simply the Lebesgue measure on $\mathbb{T}^2$.

	Assume $(f_t)_{t\in I_0}$ is a $\mathcal{C}^4$-family of transitive Anosov $\mathcal{C}^3$-diffeomorphisms of $\mathbb{T}^2$ such that $f_0=f$.
	As before, we define the vector field $X(y)=\partial_tf_t(f_t^{-1}(y))|_{t=0}$ for $y\in\mathbb{T}^2$ and $\rho_t$ the SRB measure of $f_t$.
	
	We will show three results of linear response in the case of perturbation of the cat map:
	\begin{itemize}
		\item For observables supported on a stable line;
		\item For observables supported on a line that is not parallel to the unstable direction;
		\item For a limited class of observables supported on a small circle. The observables have to cancel in a neighbourhood of all points at which the circle is tangent to the unstable direction. This is a toy model for Dirac observables on level sets around a critical point.
	\end{itemize}

	\subsection{Dirac observables along a stable line}
	
	Let $x_0\in \mathbb{T}^2$ and $E^s(x_0)$ the stable line going through $x$. Let $X_0\in\mathbb{R}^2$ be such that $\pi(X_0) = x_0$
	Without loss of generality, we assume that $X_0 = (0,0)$ and write $E^s$ for $E^s(x_0)$.

	\begin{proposition}
		\label{stable_line_prop}
		Let $V$ be a small enough open neighbourhood of $x_0$ and $h:\mathbb{T}^2\rightarrow\mathbb{R}$ be of class $\mathcal{C}^3$ be such that $\mathrm{supp}(h)\subset V$.Let $g:\mathbb{T}^2\rightarrow\mathbb{R}$ such that in the connected component of $X_0$ in $\pi^{-1}(V)$ holds $g(\pi(X_0 + (s,u))) = u$.
		Let $\mathcal{W}_0 = V\cap \{g=0\}$ and $\theta=h\delta(g)$.
		
		Then the map $t\mapsto\langle\theta,\rho_t\rangle$ is differentiable at $t=0$.
		
		 Furthermore, its derivative is the exponentially convergent sum
		\begin{equation}
		\label{result_stable_line_torus}
			\partial_t\langle\theta,\rho_t\rangle|_{t=0} = -\sum_{k=0}^\infty \nu^k\int_{f^{-k}(E^s)} h(f^k(x_s))(\mathrm{div}X)(x_s)dx_s \ ,
		\end{equation}
		
		where $dx_s$ is the Lebesgue measure on $E^s$.
	\end{proposition}

	Before moving on to the proof of this proposition, we give some brief remarks. Let $u\geq3$, $s>0$. It is known, for example through the theory of dynamical determinants \cite{baladibook16}, that $1$ is the only element of the spectrum of $\mathcal{L}_0$ on $\mathcal{B}^{u-2, s+2}$ outside of the disc of radius $r_{ess}(\mathcal{L}_0)\leq\max(\lambda^{-(u-2)},\nu^{s+2})$. 
	
	Thus, since $\Pi_0\mathrm{div(X_0)}=0$, the decay of $\langle\theta,\mathcal{L}_0^k(\mathrm{div}(X_0))\rangle$ is faster than $$r_{ess}(\mathcal{L}_0, \mathcal{B}^{u-2,s+2})^k = \max(\lambda^{-(u-2)},\nu^{s+2})^k \ .$$
	
	Since $f$ is $\mathcal{C}^\infty$, if the family $t\mapsto f_t$ is also $\mathcal{C}^\infty$, one may choose $u$ and $s$ arbitrarily large. Thus the sum in (\ref{result_stable_line_torus}) converges faster than any geometric sum.
	
	This is not contradictory with Proposition \ref{stable_line_prop}, since nothing is claimed about the speed of decay of $\int_{f^{-k}(E^s)} h(f^k(x_s))(\mathrm{div}X)(x_s)dx_s$. 
	\begin{proof}
		$\mathcal{W}_0$ is a portion of the stable manifold of $f$ at $x_0$, hence obviously   an admissible stable leaf, thus it is trivially foliated by admissible stable leaves.
		Let $g(x,y)=x-\dfrac{1+\sqrt{5}}{2}y$. Then the level sets of $g$ are the stable lines of $f$.
		
		By Theorem \ref{linear_response}, applied with $p=0$, the map $t\mapsto\langle\theta,\rho_t\rangle$ is differentiable at $0$ and:
		\begin{align*}
			\partial_t\langle\theta,\rho_t\rangle|_{t=0} &= -\sum_{k=0}^\infty \int_{\mathcal{W}_0} h(s)\mathcal{L}^k(\mathrm{div}X)(s)ds \\
			&= -\sum_{k=0}^\infty \int_{E^s} h(s)\mathcal{L}^k(\mathrm{div}X)(s)ds \\
			&= -\sum_{k=0}^\infty \int_{E^s} h(s)(\mathrm{div}X)(f^{-k}(s))ds \ ,
		\end{align*}
since $|\mathrm{Det} Df|=1$ everywhere.
		
		This integral runs along a stable line, on which $f$ acts homothetically with contraction rate $\nu$. Thus, for $k\in\mathbb{N}$, the change of variables $s=f^k(x_s)$ yields $ds=\nu^k dx_s$ and 
		\begin{equation*}
			\partial_t\langle\theta,\rho_t\rangle |_{t=0}= -\sum_{k=0}^\infty \nu^k\int_{f^{-k}(E^s)} h(f^k(x))(\mathrm{div}X)(x)dx \ , 
		\end{equation*}
		showing (\ref{result_stable_line_torus}).
	\end{proof}
	
	\begin{figure}[h]
		\centerline{\includegraphics[width= 15cm]{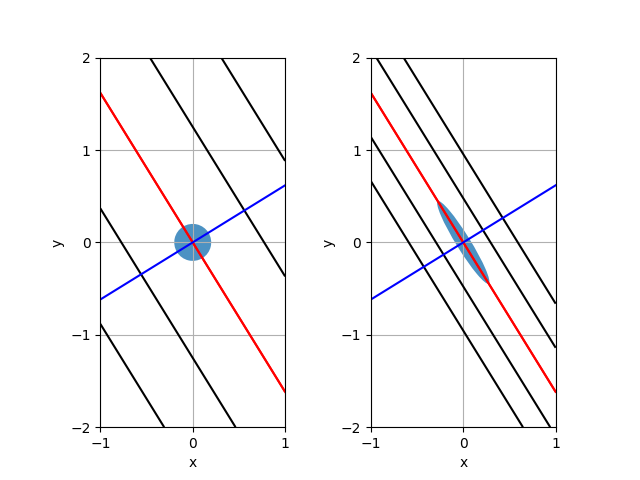}}
		\caption{Behavior of the support of a Dirac observable on the stable manifold under the reverse dynamics. The two axes through the origin are the unstable and stable directions. Black and red lines are level sets for the function $g(u,s) = u$ with the red line being the level set  $\{g=0\}$. In the left side picture, the shaded area is the support of $h$. The right side picture is the image of the left side picture by $f^{-1}$. Both pictures are in $\mathbb{R}^2$, each of the grid squares corresponds to a fundamental domain of the torus.}
	\end{figure}

	\subsection{Dirac observables along a non-unstable line}

		We now turn to observables supported on a non-unstable line.

		\begin{proposition}
			\label{generic_line_prop}
			
			Let $\alpha > 0$ and $D$ be the projection in $\mathbb{T}^2$ of the line given by the equation $u=\alpha s$, where $(s,u)\in\mathbb{R}^2$ are the coordinates in the orthonormal frame of the stable and unstable directions.
			
			Let $V$ be a small enough open neighbourhood of $0$ and $h:\mathbb{T}^2\rightarrow\mathbb{R}$ be of class $\mathcal{C}^3$ such that $\mathrm{supp}(h)\subset V$. Let $g:\mathbb{T}^2\rightarrow\mathbb{R}$ be such that in the connected component of $0$ in $\pi^{-1}(V)$ holds $g(\pi(s,u))=u-\alpha s$. 
			
			Let $\mathcal{W}_0 = V\cap D = V\cap\{g=0\}$ and let $\theta=h\delta(g)$.
			
			Then the map $t\mapsto\langle\theta,\rho_t\rangle$ is differentiable at $t=0$.
			
			Furthermore, its derivative is the exponentially convergent sum
		\begin{equation}
		\label{result_generic_line_torus}
			\partial_t\langle\theta,\rho_t\rangle|_{t=0} = -\sum_{k=0}^\infty \nu^k\sqrt{\dfrac{1+\alpha^2}{1+\alpha^2\nu^{4k}}}\int_{f^{-k}(D)} h(f^k(x))(\mathrm{div}X)(x)dx 
		\end{equation}
		where $dx$ is the Lebesgue measure on $f^{-k}(D)$ for each $k\geq 0$.
		\end{proposition}
		
		\begin{figure}[h]
			\centerline{\includegraphics[width= 15cm]{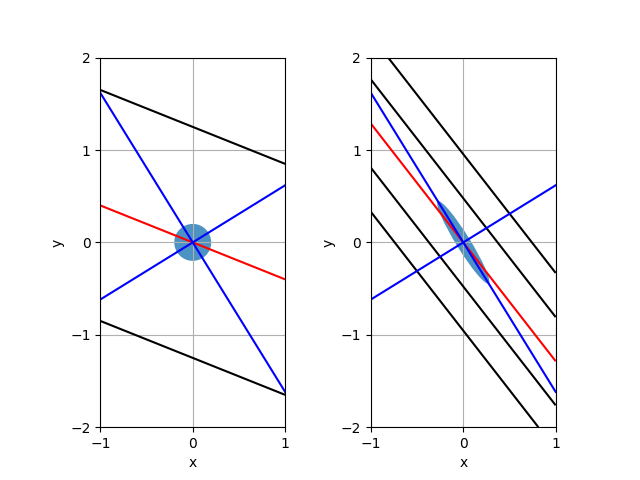}}
			\caption{Behavior of the support of a Dirac observable with $g(x,y)=y-0.4x$ under the reverse dynamics.  Black and red lines are level sets for $g$.The red line is the level set  $\{g=0\}$. The shaded area is the support of $h$. The right side picture is the image of the left side one by $f^{-1}$.}
		\end{figure}
		
		\begin{proof}
			Heuristically, the previous argument should apply whenever $D$ is in the stable cone at $0$. But, since the angle formed by the stable cone increases to be arbitrarily close to $\dfrac{\pi}{2}$ when one considers an iterate $f^m$ of $f$ with $m$ large enough, linear response may hold for observables supported on $D$ whenever it is not an unstable line. We first prove some lemmas to formalize this heuristic.
			
			\begin{lemma}
				\label{div_Xm}
				Let $m\in\mathbb{N}, m\geq 1$. Let $X_m$ be the vector field defined by $$X_m =\partial_t f^m \circ f^{-m}.$$

				Then, for all $y\in \mathbb{T}^2$:
				\begin{equation}
					(\mathrm{div}X_m)(y)=\sum_{k=0}^{m-1} (\mathrm{div}X)(f^{-k}(y))  \ .
				\end{equation}
				
			\end{lemma}
			
			\begin{proof}
				Let $y=(s,u)\in\mathbb{T}$ - with $(s,u)$ the coordinates in the stable and unstable directions.
							
				A direct computation shows that
				\begin{equation}
					\label{div_formula}
					X_m(y)=\sum_{k=0}^{m-1}Df^k(f^{-k}(y))X(f^{-k}(y))
				\end{equation}
				
				Let $X_s, X_u$ be the components of $X$ respectively in the stable and unstable directions --- which are orthogonal.
				Then, for $0\leq k \leq m-1$:
				\begin{align*}
				\mathrm{div}(Df(f^{-k}(y))X(f^{-k}(y)) &= \mathrm{div}\left[(\nu^k X_s + \lambda^k X_u)(\lambda^k s, \nu^k u)\right]\\
				&= \nu^k\lambda^k \partial_s  X_s(\lambda^k s, \nu^k u) + \lambda^k\nu^k \partial_u X_u(\lambda^k s,\nu^k u)\\
				&= \partial_s  X_s(\lambda^k s, \nu^k u) + \partial_u X_u(\lambda^k s,\nu^k u) \\
				&= (\mathrm{div}X)(f^{-k}(y))
				\end{align*}
				showing (\ref{div_formula}) by summing over $k$.
			\end{proof}
			
			\begin{lemma}
				\label{kappa_infty}
				For $m\in\mathbb{N}$ with $m\geq 1$, let $\kappa(m)$ be the maximal $\kappa$ usable in the definition of the stable cones for the $\mathcal{B}^{u,s}$ spaces adapted to $f^m$.
				Then $\kappa(m)$ is an increasing function and
				\begin{equation}
					\lim_{m\rightarrow +\infty}\kappa(m) = + \infty
				\end{equation}
			\end{lemma}
			\begin{proof}
				 Let $m\geq 1$. Following section \ref{admissible_stable_leaves_ss}, a real $\kappa>0$ is usable in the definition of the stable cones $\mathcal{C}^s(x,\kappa)$ for the $\mathcal{B}^{u,s}$ spaces adapted to $f^m$ if and only if it satisfies the two following conditions:
				 \begin{enumerate}
				 	\item \label{interior_cone} For all $x\in\mathbb{T}^2$, $D_xf^{-m}\mathcal{C}^s(x,\kappa)$ is in the interior of $\mathcal{C}^s(f^{-m}(x),\kappa)$;
				 	\item \label{kappa_bound} $(1+\kappa)^2\sqrt{1+4\kappa^2}<\nu^{-m} = \lambda^m$.
				 \end{enumerate}
				 
				 	By a direct computation, condition \ref{interior_cone} is always satisfied. Therefore $\kappa(m)$ is the solution of
				 	\begin{equation}
					 	(1+\kappa(m))^2\sqrt{1+4\kappa(m)^2} = \nu^{-m} = \lambda^m.
				 	\end{equation}
				 	Hence $\kappa(m)$ is  increasing and $$\lim_{m\rightarrow +\infty} \kappa(m) = + \infty \ .$$ 
				 
			\end{proof}
			
			We move back to the proof of Proposition \ref{generic_line_prop}.
			
			$\mathcal{W}_0$ is obviously, in the neighbourhood of each of its points, the graph of the function 
			\begin{equation}
				\chi: \left\{\begin{array}{rcl}
				(-\dfrac{\tau_i}{3\max(1,|\alpha|)}, \dfrac{\tau_i}{3\max(1,|\alpha|)})&\rightarrow& (-2\tau_i/3, 2\tau_i/3) \\
				s&\mapsto& \alpha s
				\end{array}\right.
			\end{equation}
			with $\tau_i$ defined as in section \ref{admissible_stable_leaves_ss}.
			
			Hence $$|D\chi|=|\alpha|$$ and $$|\chi|_{\mathcal{C}^4} \leq \alpha(1+\mathrm{diam}(\mathcal{W}_0)) \ .$$
			
			By Lemma \ref{kappa_infty}, we can define $m\geq 1$  such that $\kappa(m)>|\alpha|$. Then $\mathcal{W}_0$ is an admissible stable leaf for $f^m$ in the neighbourhood of each of its points, and thus trivially foliated by admissible stable leaves.
			
			By Theorem \ref{linear_response}, applied with $p=0$, and since $\rho_t$ is the SRB measure of $f_t^m$, the map $t\mapsto\langle\theta,\rho_t\rangle$ is differentiable at $0$ and:
			\begin{align*}
			\partial_t\langle\theta,\rho_t\rangle|_{t=0} &= -\sum_{k=0}^\infty \int_{\mathcal{W}_0} h(y)\mathcal{L}^{km}(\mathrm{div}X_m)(y)dy \\
			&= -\sum_{k=0}^\infty \int_{D} h(y)\mathcal{L}^{km}(\mathrm{div}X_m)(y)dy \\
			&= -\sum_{k=0}^\infty \int_{D} h(y)(\mathrm{div}X)(f^{-k}(y))dy
			\end{align*}
			by Lemma \ref{div_Xm}.
			
			For $k\geq 0$, the change of variable $y=f^k(x)$ along $D$ yields $$dy=\nu^k\sqrt{\dfrac{1+\alpha^2}{1+\alpha^2\nu^{4k}}}dx \ .$$
			Hence:
			\begin{equation}
				\partial_t\langle\theta,\rho_t\rangle|_{t=0} = -\sum_{k=0}^\infty \nu^k\sqrt{\dfrac{1+\alpha^2}{1+\alpha^2\nu^{4k}}}\int_{f^{-k}(D)} h(f^k(x))(\mathrm{div}X)(x)dx \ .
			\end{equation}
			\end{proof}
			
			\subsection{Dirac observables along a small circle}
				In this subsection, we prove linear response for a class of observables supported on a circle, canceling around the points where the circle is tangent to the unstable direction.
				
				This serves as a simplified example of the general situation where the observable is supported on a level set of a function near a local maximum or minimum.

				\begin{proposition}
					\label{circle_prop}
					
					Let $0<r<\dfrac{1}{4}$ and $C$ be the projection onto $\mathbb{T}^2$ of the circle $\tilde{C}$ of center $0$ and of radius $r$.
					
					Let $V$ be a small enough open neighbourhood of $C$ in $\mathbb{T}^2$. Let $g:\mathbb{T}^2\rightarrow\mathbb{R}$ be such that in the connected component of $\tilde{C}$ in $\pi^{-1}(V)$ holds $g(\pi(x,y)) = x^2 + y^2$.
					
					For $0<\epsilon<r$, let $$R_\epsilon=\{(s,u)\in\mathbb{T}^2, |s|^2<r^2-\epsilon^2\}$$ where $(s,u)$ are the local coordinates around $0$ in the stable and unstable directions.
					
					Let $h:\mathbb{T}^2\rightarrow\mathbb{R}$ be of class $\mathcal{C}^3$ such that there exist $0<\epsilon<r$ with $$\mathrm{supp}(h)\subset V\cap R_\epsilon$$ and let $\mathcal{W}_0=C\cap V$. Let $\theta=h\delta(g-r^2)$.
					
					Then the map $t\mapsto\langle\theta,\rho_t\rangle$ is differentiable at $t=0$.
					
					Furthermore, its derivative is the exponentially convergent sum
					\begin{align}
					\label{result_circle_torus}
					&\partial_t\langle\theta,\rho_t\rangle|_{t=0} \\
					= &-\sum_{k=0}^\infty \nu^k\int_{F^{-k}(\widetilde{C})} \sqrt{\dfrac{u^2+\nu^{4k}s^2}{u^2+\nu^{8k}s^2}}\widetilde{h}(F^k(s,u))(\mathrm{div}\widetilde{X})(s,u) d_{F^{-k}(\widetilde{C})}(s,u)
					\end{align}
					
					where $d_{f^{-k}(C)}(s,u)$ is the Lebesgue measure on $F^{-k}(C)$ for each $k\geq 0$ and $\widetilde{X}, \widetilde{h}$ lifts of $X$ and $h$ to $\mathbb{R}^2$.
				\end{proposition}
				
					\begin{figure}[h]
						\centerline{\includegraphics[width= 15cm]{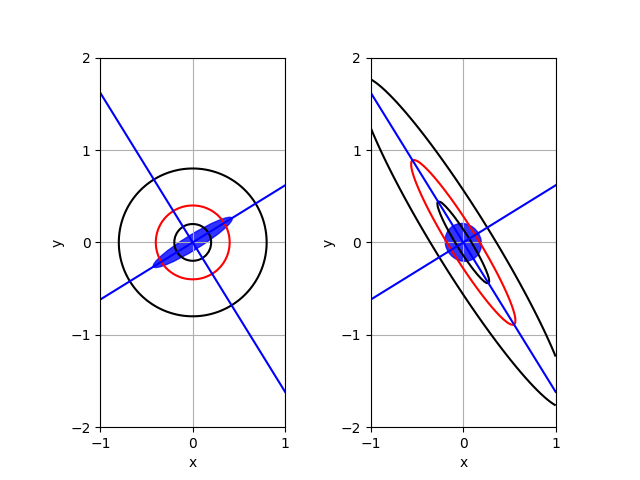}}
						\caption{Behavior of the support of a Dirac observable with $g(x,y)=x^2+y^2$ under the reverse dynamics. In the left side picture, black circles are level sets of $g$ and the red circle is the level set $\{g=0.4 ^2\}$. The shaded area is the support of $h$. Note that it excludes the region where the red circle is tangent to the unstable direction. The right side picture is the image of the left side one by $f^{-1}$.} 
					\end{figure}
				
				\begin{proof}
					Since $\mathrm{supp}(h)\subset R_\epsilon$, $\mathcal{W}_0$ is locally the graph of a function $\chi: E ^s\rightarrow E^u$ with $|D\chi|<\dfrac{r}{\epsilon}$. As in the proof of Proposition \ref{generic_line_prop}, up to considering an iterate of $f$, this shows that $W$ is foliated by admissible stable leaves.
					
					Therefore, $t\mapsto\langle\theta,\rho_t\rangle$ is differentiable at $t=0$ and:
					\begin{align*}
					\partial_t\langle\theta,\rho_t\rangle|_{t=0} &= -\sum_{k=0}^\infty \int_{C} h(y)(\mathrm{div}X)(f^{-k}(y))dy\\
					&=  -\sum_{k=0}^\infty \int_{\widetilde{C}} \widetilde{h}(s,u)(\mathrm{div}\widetilde{X})(F^{-k}(s,u))d_{\widetilde{C}}(s,u) \ .
					\end{align*}
					The change of variables $(s,u)=F^k(s,u)$ yields $$d_{\widetilde{C}}(s,u)=\nu^k\sqrt{\dfrac{1+\nu^{4k}\dfrac{s^2}{u^2}}{1+\nu^{8k}\dfrac{s^2}{u^2}}}d_{F^{-k}(\widetilde{C})}(s,u) \ .$$
					
					Therefore:
					\begin{align*}
					&\partial_t\langle\theta,\rho_t\rangle|_{t=0} \\
					=& -\sum_{k=0}^\infty \nu^k\int_{F^{-k}(\widetilde{C})} \sqrt{\dfrac{u^2+\nu^{4k}s^2}{u^2+\nu^{8k}s^2}}\widetilde{h}(F^k(s,u))(\mathrm{div}\widetilde{X})(s,u)d_{F^{-k}(\widetilde{C})}(s,u) \ .
					\end{align*}
				\end{proof}
		
	\bibliographystyle{abbrv}
	\bibliography{linear_response}

\begin{thebibliography}{10}

\bibitem{baladibook94}
V.~Baladi.
\newblock {\em Positive transfer operators and decay of correlations},
  volume~16 of {\em Advanced Series in Nonlinear Dynamics}.
\newblock World Scientific Publishing Co., Inc., River Edge, NJ, 2000.

\bibitem{baladiicm}
V.~{Baladi}.
\newblock {Linear response, or else}.
\newblock {\em Proceedings of the ICM - Seoul 2014}, (III):525--545, Aug. 2014.

\bibitem{baladibook16}
V.~Baladi.
\newblock {\em Dynamical Zeta Functions and Dynamical Determinants for
  Hyperbolic Maps}.
\newblock Book manuscript, 2016.

\bibitem{baladi17}
V.~Baladi.
\newblock The {Q}uest for the {U}ltimate {A}nisotropic {B}anach {S}pace.
\newblock {\em J. Stat. Phys.}, 166(3-4):525--557, 2017.

\bibitem{bkl16}
V.~Baladi, T.~Kuna, and V.~Lucarini.
\newblock Linear and fractional response for the {SRB} measure of smooth
  hyperbolic attractors and discontinuous observables.
\newblock {\em Nonlinearity}, 30(3):1204, 2017.
\newblock Corrigendum Nonlinearity doi.org/10.1088/1361-6544/aa7768.

\bibitem{baladi_tsuji}
V.~Baladi and M.~Tsujii.
\newblock Anisotropic {H}\"older and {S}obolev spaces for hyperbolic
  diffeomorphisms.
\newblock {\em Ann. Inst. Fourier (Grenoble)}, 57(1):127--154, 2007.

\bibitem{gl06}
S.~Gou\"ezel and C.~Liverani.
\newblock Banach spaces adapted to {A}nosov systems.
\newblock {\em Ergodic Theory Dynam. Systems}, 26(1):189--217, 2006.

\bibitem{hennion93}
H.~Hennion.
\newblock Sur un th\'eor\`eme spectral et son application aux noyaux
  lipchitziens.
\newblock {\em Proc. Amer. Math. Soc.}, 118(2):627--634, 1993.

\bibitem{kato}
T.~Kato.
\newblock {\em Perturbation theory for linear operators}.
\newblock Classics in Mathematics. Springer-Verlag, Berlin, 1995.
\newblock Reprint of the 1980 edition.

\bibitem{kh_book}
A.~Katok and B.~Hasselblatt.
\newblock {\em Introduction to the modern theory of dynamical systems},
  volume~54 of {\em Encyclopedia of Mathematics and its Applications}.
\newblock Cambridge University Press, Cambridge, 1995.
\newblock With a supplementary chapter by Katok and Leonardo Mendoza.

\bibitem{lucariniclimate17}
V.~Lucarini, F.~Ragone, and F.~Lunkeit.
\newblock Predicting {C}limate {C}hange {U}sing {R}esponse {T}heory: {G}lobal
  {A}verages and {S}patial {P}atterns.
\newblock {\em J. Stat. Phys.}, 166(3-4):1036--1064, 2017.

\bibitem{mather68}
J.~N. Mather.
\newblock Characterization of {A}nosov diffeomorphisms.
\newblock {\em Nederl. Akad. Wetensch. Proc. Ser. A 71 Indag. Math.},
  30:479--483, 1968.

\bibitem{nussbaum69}
R.~D. Nussbaum.
\newblock The radius of the essential spectrum.
\newblock {\em Duke Math. J.}, 37:473--478, 1970.

\bibitem{ruelle_76}
D.~Ruelle.
\newblock A measure associated with {A}xiom-{A} attractors.
\newblock {\em Amer. J. Math.}, 98(3):619--654, 1976.

\bibitem{ruelle09}
D.~Ruelle.
\newblock A review of linear response theory for general differentiable
  dynamical systems.
\newblock {\em Nonlinearity}, 22(4):855--870, 2009.

\bibitem{walters_book}
P.~Walters.
\newblock {\em An introduction to ergodic theory}, volume~79 of {\em Graduate
  Texts in Mathematics}.
\newblock Springer-Verlag, New York-Berlin, 1982.

\bibitem{srb_young}
L.-S. Young.
\newblock What are {SRB} measures, and which dynamical systems have them?
\newblock {\em J. Statist. Phys.}, 108(5-6):733--754, 2002.
\newblock Dedicated to David Ruelle and Yasha Sinai on the occasion of their
  65th birthdays.

\end{thebibliography}
\end{document}